\documentclass[11pt]{amsart}
\usepackage[margin=3.14cm]{geometry}
\usepackage{mathrsfs,amsmath,amsfonts,amssymb,amsthm}
\usepackage{enumitem}
\usepackage[dvipsnames,svgnames,table]{xcolor}
\usepackage[linktocpage=true,colorlinks=true,linkcolor=RubineRed!85!black,citecolor=ForestGreen,urlcolor=green,pdfborder={0 0 0}]{hyperref}

\theoremstyle{plain}
\newtheorem{theorem}{Theorem}[section]
\newtheorem{proposition}[theorem]{Proposition}
\newtheorem{lemma}[theorem]{Lemma}
\theoremstyle{definition}
\newtheorem{remark}[theorem]{Remark}
\numberwithin{equation}{section}

\newcommand*{\dif}{\mathop{}\!\mathrm{d}}
\def \R {\mathbb{R}}
\def \JJ {\mathcal{J}_\varepsilon}
\def \mm {\mathfrak{m}_\varepsilon}
\def \Pe {\Phi_\varepsilon}
\def \pe {\psi_\varepsilon}
\def \mbe {\mu_{\varepsilon,\beta}}
\def \mbew {\widetilde{\mu}_{\varepsilon,\beta}}
\def \nbe {\nu_{\varepsilon,\beta}}
\def \nbew {\widetilde{\nu}_{\varepsilon,\beta}}

\title{On the Fermi-Dirac-type Fisher information}
\begin{document}
\author{Yuzhe Zhu} 
\date{August 20, 2025}

\begin{abstract}
We consider kinetic models for Fermi-Dirac-like particles obeying the exclusion principle. A generalized notion of Fisher information, tailored to kinetic equations of Fermi-Dirac-Fokker-Planck type, is introduced via the associated entropy dissipation identity. We show that, subject to a suitable upper bound on the initial data, this quantity decreases along solutions of the Fermi-Dirac-Fokker-Planck equation, while monotonicity can fail in the absence of such a bound. We also discuss the time evolution of this Fermi-Dirac-type Fisher information for the heat equation and the linear-type Landau-Fermi-Dirac equation with Maxwell molecules. 
\end{abstract}

\maketitle
\hypersetup{bookmarksdepth=2}
\setcounter{tocdepth}{1}
\tableofcontents

\section{Introduction}
We study a generalized Fisher information adapted to kinetic equations modelling Fermi-Dirac-type systems that incorporate the exclusion principle. The primary model of concern is the space homogeneous Fermi-Dirac-Fokker-Planck equation
\begin{align}\label{FDFP}
\partial_t f=\Delta f+\nabla\cdot[v f\;\!(1-\varepsilon f)]. 
\end{align}
Here, the unknown $f= f(t,v)$ denotes the density distribution at time $t\in\R_+$ in the velocity space $v\in\R^d$. The quantum parameter $\varepsilon>0$, typically small, encodes the exclusion effect. 

The nonlinearity in \eqref{FDFP} is captured by the mobility function $\mm:[0,\varepsilon^{-1}]\to[0,\infty)$, 
\begin{align}\label{mobility}
\mm(f):= f\;\!(1-\varepsilon f),
\end{align}
which recovers the standard Ornstein-Uhlenbeck drift when $\varepsilon=0$. It accounts for the exclusion principle by enforcing the microscopic constraint  
\begin{align*}
0\le f\le \varepsilon^{-1}. 
\end{align*}
The basic properties of \eqref{FDFP}, including well-posedness and long-time asymptotics, can be found in \cite{CLR}. 

The development of kinetic models for systems exhibiting Fermi-Dirac-like statistics dates back to \cite{Nordhiem,UU,Lynden}, and has since expanded to include Fokker-Planck variants \cite{KQ,Frank}. The systems of fermions obey Pauli's exclusion principle, which limits the maximum mean occupation numbers of each quantum state. This exclusion effect also extends beyond the quantum systems to model other phenomena involving saturation, such as excluded volume interactions and crowding dynamics. Specific applications of Fermi-Dirac-Fokker-Planck-type equations appear in various contexts including phase segregation in lattice gases \cite{GL} and chemotaxis in bacterial populations \cite{CRRS}; see \cite{Chavanis} for a broader review. 

\subsection{Entropy structure}
The quantum mechanical entropy $\mathcal{E}_\varepsilon$, introduced to take the exclusion principle into account (see, for example, the exposition in \cite{LL}), is defined by
\begin{align}\label{EEntropy}
\begin{aligned}
\mathcal{E}_\varepsilon(f)&:=\int_{\R^d}U_\varepsilon(f) \dif v,\\
U_\varepsilon(f)&:=\frac{1}{\varepsilon}\left[\varepsilon f\log (\varepsilon f)+(1-\varepsilon f) \log (1-\varepsilon f)\right]. 
\end{aligned}
\end{align}
The functional $\mathcal{E}_\varepsilon$ is also referred to as the internal energy associated with the function $U_\varepsilon$. It satisfies $U_\varepsilon''(f)=\mm(f)^{-1}$ so that $\mathcal{E}_\varepsilon$ is convex for $f$ taking values in $(0,\varepsilon^{-1})$. The Fermi-Dirac-Fokker-Planck evolution \eqref{FDFP} decreases the free-energy functional
\begin{align}\label{HHE}
\mathcal{H}_\varepsilon(f):= \mathcal{E}_\varepsilon(f) + \frac{1}{2}\int_{\R^d}|v|^2 f \dif v, 
\end{align}
where the second term on the right-hand side corresponds to the potential energy associated with the confining drift. The first variation of $\mathcal{H}_\varepsilon$ is given by 
\begin{align*}
\frac{\delta\mathcal{H}_\varepsilon}{\delta f}=U_\varepsilon'(f)+\frac{|v|^2}{2}=\log\frac{f}{1-\varepsilon f}+\frac{|v|^2}{2}. 
\end{align*}
In terms of the mobility function \eqref{mobility} and this functional derivative, \eqref{FDFP} can be equivalently written as  
\begin{align}\label{FDFP-E}
\partial_t f=\nabla\cdot\left[\mm(f)\,\nabla\;\!\frac{\delta\mathcal{H}_\varepsilon}{\delta f}\right]. 
\end{align}
This formulation reveals a gradient flow structure characterized by the functional $\mathcal{H}_\varepsilon$ and the nonlinear mobility $\mm$; see \S\!~\ref{flow} below for further discussion. 

\subsection{Entropy dissipation}
In view of \eqref{FDFP-E} and a straightforward integration by parts, one finds that the time derivative of the free energy $\mathcal{H}_\varepsilon$ along a solution $f$ of \eqref{FDFP-E} is given by 
\begin{align}\label{FDFP-I}
\frac{\dif}{\dif t}\,\mathcal{H}_\varepsilon(f)=\int_{\R^d}\frac{\delta\mathcal{H}_\varepsilon}{\delta f}\,\partial_t f \dif v
=-\int_{\R^d}\mm(f)\left|\nabla\;\!\frac{\delta\mathcal{H}_\varepsilon}{\delta f}\right|^2 \dif v. 
\end{align}
The Fermi-Dirac-type Fisher information is then defined as the dissipation rate of $\mathcal{H}_\varepsilon$ along the evolution, given by
\begin{align}\label{JJ}
\JJ(f):=\int_{\R^d} \mm(f)\left|\nabla\;\!\frac{\delta\mathcal{H}_\varepsilon}{\delta f}\right|^2 \dif v
=\int_{\R^d}f(1-\varepsilon f)\left|\frac{\nabla f}{f(1-\varepsilon f)}+v\right|^2 \dif v.
\end{align}
Based on the entropy dissipation \eqref{FDFP-I}, this identifies the natural quantum analogue of the classical Fisher information, incorporating the exclusion effect through the nonlinear mobility.

The equilibrium states, which are the free energy minimizers, take the form of the Fermi-Dirac distribution $\mbe$ with $\beta>0$, 
\begin{align}\label{mbe}
\mbe(v):=\frac{1}{\varepsilon+\beta\;\!e^{\;\!|v|^2/2}},
\end{align}
which are stationary solutions of \eqref{FDFP} and satisfy $\JJ(\mbe)=0$. The constant $\beta$ can be uniquely determined by the conserved total mass of densities, yielding a unique Fermi-Dirac equilibrium. 

\subsection{Relative form}\label{relative}
The free energy generically admits an equivalent description as a relative entropy with respect to a reference density. Concretely, for any two densities $f$ and $g$ valued in $(0,\varepsilon^{-1})$, one defines the Bregman divergence, 
\begin{align*}
\mathcal{E}_\varepsilon(f\;\!|\;\!g):=\mathcal{E}_\varepsilon(f)-\mathcal{E}_\varepsilon(g)-\left\langle\frac{\delta\mathcal{E}_\varepsilon}{\delta f}(g),\,f-g\right\rangle
=\int_{\R^d}\left[U_\varepsilon(f)-U_\varepsilon(g)-U_\varepsilon'(g)\,(f-g)\right] \dif v. 
\end{align*}
Now that $U_\varepsilon$ is convex, it measures the excess of $\mathcal{E}_\varepsilon(f)$ over its linear approximation at $g$. When the reference density is chosen to be the Fermi-Dirac equilibrium $\mbe$ with $\beta>0$ so that $\mbe$ carries the same mass as $f$, the identity $U_\varepsilon'(\mbe)=-\log\beta-\frac{|v|^2}{2}$ yields that  
\begin{align*}
\mathcal{E}_\varepsilon(f\;\!|\;\!\mbe)=\mathcal{H}_\varepsilon(f)-\mathcal{H}_\varepsilon(\mbe)+\log\beta\int_{\R^d}(f-\mbe)\dif v=\mathcal{H}_\varepsilon(f)-\mathcal{H}_\varepsilon(\mbe), 
\end{align*}
where the constant $\mathcal{H}_\varepsilon(\mbe)$ is the minimum free energy under the mass constraint. Thus, in mass-conserving dynamics, one may track either the free energy or the relative entropy. 

In this sense, the functional $\JJ$, interpreted as the entropy dissipation relative to Fermi-Dirac equilibria, can be regarded as the relative form of the Fermi-Dirac Fisher information $\mathcal{I}_\varepsilon$, which is defined by 
\begin{align}\label{II}
\mathcal{I}_\varepsilon(f):=\int_{\R^d} \mm(f) \left|\nabla\frac{\delta\mathcal{E}_{\varepsilon}}{\delta f}\right|^2 \dif v
=\int_{\R^d} \frac{|\nabla f|^2}{ f(1-\varepsilon f)} \dif v. 
\end{align}
The functional $\mathcal{I}_\varepsilon$ also arises naturally as the Fermi-Dirac entropy dissipation along the heat flow in $\R^d$. 

\subsection{Heat flow}
Prior to stating the main results for Fermi-Dirac-Fokker-Planck equation, we examine the Fermi-Dirac Fisher information for the heat equation on general manifolds as a preliminary case. 

Let $M$ be a smooth manifold without boundary equipped with a Riemannian metric $\langle\cdot,\cdot\rangle_M$. We denote by $|\cdot|_M$ the norm induced by the metric, and by $\Delta_M$ the Laplace-Beltrami operator on $M$. Consider a function $f:\R_+\times M\to(0,\varepsilon^{-1})$, for some fixed $\varepsilon>0$, satisfying the heat equation 
\begin{align}\label{heat}
\partial_t f=\Delta_M f. 
\end{align}
It can be rewritten as 
\begin{align}\label{heat-E}
\partial_t f={\rm div}\left[\mm(f)\,\nabla\;\!\frac{\delta\mathcal{E}_{\varepsilon,M}}{\delta f}\right] 
{\quad\rm for\quad} \frac{\delta\mathcal{E}_{\varepsilon,M}}{\delta f}=\log\frac{f}{1-\varepsilon f}, 
\end{align}
where $\mathcal{E}_{\varepsilon,M}$ denotes the Fermi-Dirac entropy on $M$, given by the same formula as $\mathcal{E}_\varepsilon$ in \eqref{EEntropy} with $\R^d$ and $\dif v$ replaced by $M$ and its volume form $\dif V_M$, respectively. 

The entropy dissipation along the heat flow yields the associated Fisher information. Specifically, by \eqref{heat-E} and integration by parts, we have
\begin{align*}
\frac{\dif}{\dif t}\,\mathcal{E}_{\varepsilon,M}(f)
=\int_M \frac{\delta\mathcal{E}_{\varepsilon,M}}{\delta f}\,\partial_tf \dif V_M
=-\int_M \mm(f) \left|\nabla\frac{\delta\mathcal{E}_{\varepsilon,M}}{\delta f}\right|_M^2 \dif V_M. 
\end{align*}
Accordingly, we introduce the Fermi-Dirac Fisher information for the heat equation on $M$ as 
\begin{align}\label{IIM}
\mathcal{I}_{\varepsilon,M}(f):=\int_M \mm(f) \left|\nabla\frac{\delta\mathcal{E}_{\varepsilon,M}}{\delta f}\right|_M^2 \dif V_M
=\int_M \frac{|\nabla f|^2}{ f(1-\varepsilon f)} \dif V_M.  
\end{align}
This functional reduces to the classical Fisher information when $\varepsilon=0$. 

Recall from the discussion in \S\!~\ref{relative} that the formula \eqref{JJ} for $\JJ$, expressed in terms of the free energy $\mathcal{H}_\varepsilon$, represents the relative form of the Fermi-Dirac Fisher information. By contrast, the definition \eqref{IIM} of $\mathcal{I}_{\varepsilon,M}$ involves the entropy functional $\mathcal{E}_{\varepsilon,M}$, since the Fermi-Dirac distribution like $\mbe$ is generally not an equilibrium state for the heat equation. 

\subsection{Main results}
\subsubsection{Heat equation}
The Fermi-Dirac Fisher information exhibits monotonicity along the heat flow on manifolds with non-negative Ricci curvature, analogous to the classical Fisher information. More precisely, we have the following result. 
\begin{theorem}\label{thm-heat}
For any $\varepsilon>0$ and any solution $f:\R_+\times M\to(0,\varepsilon^{-1})$ of \eqref{heat}, we have 
\begin{align*}
\frac{1}{2}\frac{\dif}{\dif t}\,\mathcal{I}_{\varepsilon,M}(f)
= -\int_M \mm(f) \big[|D^2\pe|_{\rm HS}^2 +\mathrm{Ric}\;\!(\nabla\pe,\nabla\pe)\big] + \varepsilon\;\!|\nabla f|_M^2|\nabla\pe|_M^2 \dif V_M, 
\end{align*}
where we abbreviate $\pe:=\log f-\log (1-\varepsilon f)$ and  $|\cdot|_{\rm HS}$ denotes the Hilbert-Schmidt norm. 
In particular, if the Ricci curvature of $M$ satisfies $\mathrm{Ric}\ge c_M$ for some constant $c_M\in\R$, then 
\begin{align*}
\frac{1}{2}\frac{\dif}{\dif t}\,\mathcal{I}_{\varepsilon,M}(f) \le - c_M\,\mathcal{I}_{\varepsilon,M}(f). 
\end{align*}
\end{theorem}

\subsubsection{Fermi-Dirac-Fokker-Planck equation}
We now investigate the monotonicity properties of the Fermi-Dirac-type Fisher information along solutions of \eqref{FDFP}. In contrast to the case of the heat flow, the Fermi-Dirac Fisher information $\JJ$ is not always monotone along the Fermi-Dirac-Fokker-Planck flow. For any quantum parameter $\varepsilon>0$, one can construct data for which the Fermi-Dirac Fisher information increases. Despite this general lack of monotonicity, the monotonicity of $\JJ$, as well as its exponential decay, can be recovered once the initial data satisfies a pointwise upper bound condition. 

\begin{theorem}\label{thm-FP}
Let $f:\R_+\times\R^d\to(0,\varepsilon^{-1})$ be a solution to \eqref{FDFP}. 
\begin{enumerate}[label=(\roman*), leftmargin=*, labelsep=0.5em]
\item\label{FD-c} For any $\varepsilon>0$, there exists $\alpha\ge1$ and $u\in\R^d$ such that 
\begin{align*}
\frac{\dif}{\dif t}\bigg|_{t=0}\JJ(f)>0
{\quad\rm for\quad}
\left.f\right|_{t=0}\!(v)=\frac{1}{\varepsilon+e^{\;\!\alpha\;\!|v-u|^2/2}}. 
\end{align*}
\item\label{FD-d} If $0\le\left.f\right|_{t=0}\le\mbe$ in $\R^d$ with the constants $\varepsilon,\beta$ satisfying $0\le4\varepsilon\le\beta$, then we have
\begin{align*}
\frac{1}{2}\frac{\dif}{\dif t}\,\JJ(f)
\le -\left(1-\frac{4\varepsilon}{\beta}\right)\JJ(f). 
\end{align*}
\end{enumerate}
\end{theorem}

\begin{remark}
When $\varepsilon=0$, the estimates in Theorem~\ref{thm-heat} for the heat equation and in part~\ref{FD-d} of Theorem~\ref{thm-FP} for the Fokker-Planck equation coincide with the classical inequalities from the McKean-Toscani lemma \cite{McKean,Toscani}. Their arguments are based on the convexity of the Fisher information and the structure of the fundamental solutions. An alternative proof via direct computation can be found in \cite{Villani-landau-fisher}, and a related general formalism was developed in \cite{BE}. 
\end{remark}

\begin{remark}
The monotonicity of $\JJ$ along solutions of \eqref{FDFP} can also be ensured under a more refined condition, weaker than the pointwise upper bound in point~\ref{FD-d} of Theorem~\ref{thm-FP}; see Lemma~\ref{lemma-est}. 
\end{remark}

\begin{remark}
In Section~\ref{sec-LFP}, we discuss a linear-type Landau-Fermi-Dirac equation, for which Theorem~\ref{prop-landau} provides an analogue of Theorem~\ref{thm-FP}. 
\end{remark}

\subsection{Geometric structure}\label{flow}
Let us briefly mention another characterization of \eqref{FDFP} without delving into the details needed for full rigour. It exposes the structure of gradient flow encoded in the relation between \eqref{HHE} and \eqref{FDFP-E}. While not essential to the main results, it provides geometric insight into the formulation~\eqref{FDFP-E} and the entropy dissipation identity \eqref{FDFP-I}, and clarifies a role of the Fermi-Dirac-type Fisher information. This perspective also suggests possible generalizations within the framework of weighted Wasserstein geometry. 

Following Otto's calculus (see  \cite{Otto,CLSS}), one equips the space of densities $f$ taking values in $(0,\varepsilon^{-1})$ with a Riemannian metric weighted by the concave mobility function $\mm$. The norm of a tangent vector $p$ at a density $f$ is given by
\begin{align*}
\|p\|^2:=\inf\left\{ \int_{\R^d}\mm(f)\;\!|\nabla q|^2 \dif v :\ p+\nabla\cdot[\mm(f)\nabla q]=0\right\}, 
\end{align*}
which induces an inner product $\langle\!\!\!\!\;\langle\cdot,\cdot\rangle\!\!\!\!\;\rangle$ on tangent vectors. Denote by $\nabla\!\!\!\!\nabla$ the (Wasserstein) gradient and by $\nabla\!\!\!\!\nabla^2$ the corresponding Hessian. With respect to this metric, \eqref{FDFP-E} represents the gradient flow of $\mathcal{H}_\varepsilon$ and can be also written as
\begin{align*}
\partial_tf=-\nabla\!\!\!\!\nabla\mathcal{H}_\varepsilon(f). 
\end{align*} 
The Fermi-Dirac Fisher information $\JJ(f)$ is in fact the squared norm of the gradient, appearing as the dissipation of $\mathcal{H}_\varepsilon$ along the flow, 
\begin{align*}
\frac{\dif}{\dif t}\,\mathcal{H}_\varepsilon(f)=-\|\nabla\!\!\!\!\nabla\mathcal{H}_\varepsilon(f)\|^2
=-\int_{\R^d} \mm(f)\left|\nabla\;\!\frac{\delta\mathcal{H}_\varepsilon}{\delta f}\right|^2 \dif v. 
\end{align*}
Formally differentiating once more in time gives the Fisher dissipation, 
\begin{align*}
\frac{\dif^2}{\dif t^2}\mathcal{H}_\varepsilon(f) 
= -\frac{\dif}{\dif t}\, \|\nabla\!\!\!\!\nabla\mathcal{H}_\varepsilon(f)\|^2
= 2\left\langle\!\!\left\langle \nabla\!\!\!\!\nabla^2\mathcal{H}_\varepsilon(f)\,\nabla\!\!\!\!\nabla\mathcal{H}_\varepsilon(f),\,\nabla\!\!\!\!\nabla\mathcal{H}_\varepsilon(f) \right\rangle\!\!\right\rangle. 
\end{align*}
Hence the monotonic decay of the Fermi-Dirac Fisher information is linked to  the notion of geodesic convexity of $\mathcal{H}_\varepsilon$, which corresponds to the positivity of $\nabla\!\!\!\!\nabla^2\mathcal{H}_\varepsilon$ along geodesics (see \cite{McCann,CLSS} for details). Recall that $\mathcal{H}_\varepsilon$ in our case \eqref{HHE} splits into the internal energy $\mathcal{E}_\varepsilon$ and a potential energy. It has been noticed in \cite{CLSS} that the internal energy $\mathcal{E}_\varepsilon$ is convex in the weighted metric, whereas the nonlinear mobility $\mm$ may spoil the geodesic convexity of potential energies. 

\subsection{Organization of the paper}
In Section~\ref{sec-heat}, we establish the monotonicity properties of the Fermi-Dirac Fisher information for the heat equation and a related model equation. The proofs of the main results concerning the Fermi-Dirac-Fokker-Planck equation are presented in Section~\ref{sec-FP}. We conclude in Section~\ref{sec-LFP} with a discussion of the linear-type Landau-Fermi-Dirac equation based on the results from the preceding sections. 

\section{Fermi-Dirac Fisher information along heat flows}\label{sec-heat}
This section proves Theorem~\ref{thm-heat} and examines a model equation related to the heat equation on the sphere. To begin, we introduce some convenient notation. For a function $f$ on the manifold $M$ with values in $(0,\varepsilon^{-1})$, we recall the definitions
\begin{align*}
\mm&=\mm(f)=f\;\!(1-\varepsilon f),\\
\pe&=\pe(f)=\log\frac{f}{1-\varepsilon f}.
\end{align*}
Throughout this section, we write $\Delta=\Delta_M$ for the Laplace-Beltrami operator on $M$, and denote by $\langle\cdot,\cdot\rangle_M$ the Riemannian metric on $M$ with the induced norm $|\cdot|=|\cdot|_M$.
By definition, we have 
\begin{align}\label{hpsi-heat}
\begin{aligned}
&\mm\pe'=1\\
&\nabla\mm = \mm'\;\!\nabla f= \mm\mm'\;\!\nabla\pe, \\
&\langle\nabla\mm,\nabla\pe\rangle_M = \mm\mm'\;\!|\nabla\pe|^2, \\
&\Delta\mm=\mm'\;\!\Delta f+\mm''\;\!|\nabla f|^2. 
\end{aligned}
\end{align}
The Fermi-Dirac Fisher information for the heat equation, defined in \eqref{IIM}, is recast as
\begin{align*}
\mathcal{I}_{\varepsilon,M}(f)=\int_M \mm\;\!|\nabla\pe|^2 \dif V_M. 
\end{align*}
Henceforth, the volume form $\dif V_M$ is omitted from all integrals, provided no ambiguity arises. 

The first variation of $\mathcal{I}_{\varepsilon,M}$ can be calculated by taking the Gateaux derivative and applying integration by parts. Indeed, by \eqref{hpsi-heat}, we see that for any $g\in C_c^\infty(M)$,  
\begin{align*}
\left\langle\frac{\delta\mathcal{I}_{\varepsilon,M}}{\delta f},\,g\right\rangle
&=\int_M g\,\mm'\;\!|\nabla\pe|^2 + 2\;\!\mm\langle\nabla\pe,\nabla(g\;\!\pe')\rangle_M\\
&=\int_M g\,\mm'\;\!|\nabla\pe|^2 - 2g\,\Delta\pe -2g\,\pe'\langle\nabla\mm,\nabla\pe\rangle_M\\
&=-\int_M g\,\mm'\;\!|\nabla\pe|^2 + 2g\,\Delta\pe,
\end{align*}
which implies that 
\begin{align}\label{DII}
\frac{\delta\mathcal{I}_{\varepsilon,M}}{\delta f}=-\mm'\;\!|\nabla\pe|^2 - 2\;\!\Delta\pe.
\end{align}

\subsection{Derivative of Fermi-Dirac Fisher information along the heat flow} 
The formula in Theorem~\ref{thm-heat} is obtained by direct computation and simplification via integration by parts. 
\begin{proof}[Proof of Theorem~\ref{thm-heat}]
By \eqref{DII}, we have that for $f:\R_+\times M\to(0,\varepsilon^{-1})$ satisfying \eqref{heat}, 
\begin{align}\label{d-heat}
\frac{\dif}{\dif t}\,\mathcal{I}_{\varepsilon,M}(f)
=\int_M \frac{\delta\mathcal{I}_{\varepsilon,M}}{\delta f}\, \partial_t f
=-\int_M \mm'\;\!|\nabla\pe|^2\;\!\Delta f -2 \int_M \Delta\pe\,  \Delta f.
\end{align}
By \eqref{hpsi-heat} and integration by parts, we recast the first term on the right-hand side of \eqref{d-heat} as
\begin{align}\label{d-heat1}
\begin{aligned}
\int_M \mm'\;\!|\nabla\pe|^2\Delta f
&= \int_M |\nabla\pe|^2\;\!\Delta\mm-\mm''\;\!|\nabla f|^2|\nabla\pe|^2\\
&= \int_M \mm\;\! \Delta |\nabla\pe|^2 - \mm''\;\!|\nabla f|^2|\nabla\pe|^2.
\end{aligned}
\end{align}
Applying \eqref{hpsi-heat} and integration by parts for the second term of \eqref{d-heat} yields that 
\begin{align*}
\int_M \Delta \pe\,\Delta f
= \int_M - \mm\langle\nabla \Delta\pe,\nabla\pe\rangle_M. 
\end{align*}
Combining this with \eqref{d-heat} and \eqref{d-heat1}, we derive 
\begin{align}\label{II-heat}
\frac{1}{2}\frac{\dif}{\dif t}\,\mathcal{I}_{\varepsilon,M}(f)
= - \int_M \frac{1}{2}\,\mm\;\!\Delta |\nabla\pe|^2 - \mm\langle\nabla \Delta\pe,\nabla\pe\rangle_M - \frac{1}{2}\mm''\;\!|\nabla f|^2|\nabla\pe|^2.
\end{align}
It then follows from Bochner's formula that 
\begin{align*}
\frac{1}{2}\frac{\dif}{\dif t}\,\mathcal{I}_{\varepsilon,M}(f)
= - \int_M \mm\;\!|D^2\pe|_{\rm HS}^2 +\mm\;\!\mathrm{Ric}\left(\nabla\pe,\nabla\pe\right) - \frac{1}{2}\mm''\;\!|\nabla f|^2|\nabla\pe|^2. 
\end{align*}
Noticing that $\mm''=-2\varepsilon$, we see the formula for the time derivative of $\mathcal{I}_{\varepsilon,M}(f)$ as claimed. 

This further implies that 
\begin{align*}
\frac{1}{2}\frac{\dif}{\dif t}\,\mathcal{I}_{\varepsilon,M}(f)
\le - \int_M \mm\;\!\mathrm{Ric}\left(\nabla\pe,\nabla\pe\right) . 
\end{align*}
If the Ricci curvature of $M$ is bounded below by a constant $c_M$, then the desired estimate follows.
\end{proof}

\subsection{Spherical heat equation model}
Since heat flows do not admit Fermi-Dirac-type equilibria, we focus here on a heat equation model with spherical diffusion in $\R^d$. To highlight the connection with the Fermi-Dirac-Fokker-Planck equation, which will be discussed in the next section, we consider the relative form of the Fermi-Dirac Fisher information \eqref{JJ}, reformulated as
\begin{align*}
\JJ(f)=\int_{\R^d} \mm\;\!|\nabla\pe+v|^2 \dif v. 
\end{align*}
Let the dimension $d\ge2$, and let the generators of rotations be defined by 
\begin{align}\label{rotation}
\Omega_{jk}:=v_j\partial_{v_k}-v_k\partial_{v_j} {\quad\rm for\quad}j,k=1,\ldots,d. 
\end{align}
We consider the model equation 
\begin{align}\label{model}
\partial_t f= \sum\nolimits_{j<k}\Omega_{jk}^2\;\!f. 
\end{align}
Upon restricting functions to the unit sphere $\mathbb{S}^{d-1}$, the operator $\sum_{j<k}\Omega_{jk}^2$ is nothing but the Laplace-Beltrami operator $\Delta_{\mathbb{S}^{d-1}}$. In this setting, the Fermi-Dirac distribution $\mbe$ is a stationary solution of \eqref{model}. 

It turns out that the functional $\JJ$ is not necessarily decreasing along solutions of \eqref{model}. 

\begin{proposition}\label{prop-model}
Let $d\ge2$, and let $f:\R_+\times\R^d\to(0,\varepsilon^{-1})$ be a solution to \eqref{model}. 
\begin{enumerate}[label=(\roman*), leftmargin=*, labelsep=0.5em]
\item\label{model1} For any $\varepsilon,\alpha>0$, there exists $u\in\R^d$ such that 
\begin{align*}
\frac{\dif}{\dif t}\bigg|_{t=0}\JJ(f)>0
{\quad\rm for\quad}
\left.f\right|_{t=0}\!(v)=\frac{1}{\varepsilon+e^{\;\!\alpha\;\!|v-u|^2/2}}. 
\end{align*}
\item\label{model2} Let $\Pe:=\pe+\frac{|v|^2}{2}$. It holds that 
\begin{align*}
\frac{1}{2}\frac{\dif}{\dif t}\,\JJ(f)
\le \varepsilon  \int_{\R^d} |v|^4\;\!\mm^2\,|\nabla\Pe|^2  \dif v 
\le \varepsilon\sup\nolimits_{\,\R^d}\left(|v|^4\;\!\mm\right)\JJ(f). 
\end{align*}
\end{enumerate}
\end{proposition}

Let us first derive the derivative of $\JJ$ along solutions of \eqref{model}. 
\begin{lemma}\label{lemma-model}
For any solution $f:\R_+\times\R^d\to(0,\varepsilon^{-1})$ of \eqref{model}, we have 
\begin{align}\label{JJ-model}
\frac{1}{2}\frac{\dif}{\dif t}\,\JJ(f)
= -\sum_{j<k} \int_{\R^d} \mm|\nabla\Omega_{jk}\;\!\pe|^2+\varepsilon \left(|\nabla\pe|^2-|v|^2\right) |\Omega_{jk} f|^2  \dif v. 
\end{align}
\end{lemma}

\begin{proof}
By definition and integration by parts, we have 
\begin{align*}
\JJ(f)&= \int_{\R^d}\left(\mm|\nabla\pe|^2+2\;\!v\cdot\nabla f +|v|^2\;\!\mm\right)\dif v\\
&= \mathcal{I}_\varepsilon(f) -2d\int_{\R^d} f \dif v+\int_{\R^d} |v|^2\;\!\mm  \dif v. 
\end{align*}
Similarly to the last identity in \eqref{hpsi-heat}, we observe
\begin{align*}
\partial_t\;\!\mm=\sum\nolimits_{j<k}\Omega_{jk}^2\;\!\mm -\mm''\;\!|\Omega_{jk} f|^2. 
\end{align*} 
It then follows from integration by parts that 
\begin{align*}
\frac{\dif}{\dif t}\,\JJ(f) = \frac{\dif}{\dif t}\,\mathcal{I}_\varepsilon(f) - \sum_{j<k}\int_{\R^d}\mm''\;\!|v|^2|\Omega_{jk} f|^2  \dif v.
\end{align*}
The same derivation as \eqref{II-heat} yields
\begin{align*}
\frac{1}{2}\frac{\dif}{\dif t}\,\mathcal{I}_\varepsilon(f)
=-\sum_{j<k} \int_{\R^d} \frac{1}{2}\,\mm\;\!\Omega_{jk}^2 |\nabla\pe|^2 - \mm\nabla \Omega_{jk}^2\,\pe\cdot\nabla\pe - \frac{1}{2}\mm''\;\!|\Omega_{jk} f|^2\,|\nabla\pe|^2  \dif v.  
\end{align*}
The version of Bochner's formula relevant in this context reads
\begin{align*}
\frac{1}{2}\,\Omega_{jk}^2 |\nabla\pe|^2 - \nabla\Omega_{jk}^2\,\pe\cdot\nabla\pe = |\nabla\Omega_{jk}\;\!\pe|^2. 
\end{align*}
Combining the above three identities, we arrive at the desired result. 
\end{proof}

It remains to show that the Fisher dissipation given in \eqref{JJ-model} does not have a definite sign.  
\begin{proof}[Proof of Proposition~\ref{prop-model}]
Let the initial distribution ($t=0$) be of the Fermi-Dirac type, with $\alpha>0$ and $u\in\R^d$ to be specified later, that is, 
\begin{align*}
f(v)&=\frac{1}{\varepsilon+e^{\;\!\alpha\;\!|v-u|^2/2}} ,\\
\mm(v)&=f(1-\varepsilon f) =\frac{e^{\;\!\alpha\;\!|v-u|^2/2}}{(\varepsilon+e^{\;\!\alpha\;\!|v-u|^2/2})^2},\\
\pe(v)&=\log\frac{f}{1-\varepsilon f}=-\frac{\alpha\;\!|v-u|^2}{2}.
\end{align*}
A direct computation shows that
\begin{align*}
&\Omega_{jk}\pe=\alpha\;\!(v_ju_k-v_ku_j),\qquad \nabla\Omega_{jk}\pe=\alpha\;\!(e_ju_k-e_ku_j),\\
&\nabla\pe=-\alpha\;\!(v-u),\qquad \Omega_{jk} f=\alpha\;\!(v_ju_k-v_ku_j)\;\!\mm. 
\end{align*}
By applying \eqref{JJ-model} of Lemma~\ref{lemma-model}, performing a change of variables $w=v-u$, and taking into account that the odd part of the integrand vanishes upon integration, we deduce that 
\begin{align*}
\frac{1}{2}\frac{\dif}{\dif t}\bigg|_{t=0}\JJ(f)
&= -\alpha^2\sum_{j<k} \int_{\R^d} \left(u_j^2+u_k^2\right)\mm + \varepsilon\big(\alpha^2|v-u|^2-|v|^2\big) (v_ju_k-v_ku_j)^2\;\!\mm^2 \dif v\\
&=-\alpha^2\int_{\R^d} (d-1)\;\!|u|^2\mm + \varepsilon \big(\alpha^2|v-u|^2-|v|^2\big) \left[|v|^2|u|^2-(v\cdot u)^2\right] \mm^2 \dif v\\
&=-\alpha^2(d-1)\;\!|u|^2\!\int_{\R^d} \frac{e^{\;\!\alpha\;\!|w|^2/2}}{(\varepsilon+e^{\;\!\alpha\;\!|w|^2/2})^2} + \frac{\varepsilon\big[(\alpha^2-1)|w|^4-|u|^2|w|^2\big]e^{\;\!\alpha\;\!|w|^2}}{d\,(\varepsilon+e^{\;\!\alpha\;\!|w|^2/2})^4} \dif w. 
\end{align*}
We observe that the integrand above tends to negative infinity as $|u|\to\infty$. Consequently, for any $d>1$ and $\varepsilon,\alpha>0$, we can find $u\in\R^d$ with sufficiently large magnitude to make the above time derivative positive, thereby establishing part~\ref{model1}.

We conclude with a simple estimate for \eqref{JJ-model}. Substituting the identities
\begin{align*}
\nabla\pe&=\nabla\Pe-v,\\
\Omega_{jk} f&=\mm\;\!\Omega_{jk}\pe=\mm\;\!\Omega_{jk}\Pe, 
\end{align*}
into \eqref{JJ-model} gives 
\begin{align*}
\frac{1}{2}\frac{\dif}{\dif t}\,\JJ(f)
\le -\varepsilon \sum_{j<k} \int_{\R^d} \left(|\nabla\Pe|^2-2v\cdot\nabla\Pe\right)\mm^2\,|\Omega_{jk}\Pe|^2  \dif v. 
\end{align*}
The claimed result in part~\ref{model2} then follows from the Cauchy-Schwarz inequality. 
\end{proof}
 
\section{Fermi-Dirac-Fokker-Planck equation}\label{sec-FP}
This section establishes Theorem~\ref{thm-FP}, which addresses the monotonicity of the Fermi-Dirac Fisher information for \eqref{FDFP}. For clarity and reference, we first recall the relevant notation. By setting 
\begin{align*}
&\mm=\mm(f)=f(1-\varepsilon f),\\
&\Pe=\Pe(f)=\frac{\delta\mathcal{H}_\varepsilon}{\delta f}=\log\frac{f}{1-\varepsilon f}+\frac{|v|^2}{2}, 
\end{align*}
for $f:\R^d\to(0,\varepsilon^{-1})$, one readily verifies, similarly to \eqref{hpsi-heat}, that 
\begin{align}\label{hpsi}
\begin{aligned}
&\mm\Pe'=1,\\
&\nabla\mm\cdot\nabla\Pe = \mm\mm'\;\!(\nabla\Pe-v)\cdot\nabla\Pe.
\end{aligned}
\end{align}
For \eqref{FDFP} and \eqref{FDFP-E}, we define the Fermi-Dirac-Fokker-Planck operator
\begin{align}\label{LFDE}
\mathscr{L}_\varepsilon f:=\Delta f+ \nabla\cdot(v\;\!\mm)=\nabla\cdot(\mm\nabla\Pe). 
\end{align}
The Fermi-Dirac Fisher information functional $\JJ$ can be written as 
\begin{align*}
\JJ(f)=\int_{\R^d} \mm\;\!|\nabla\Pe|^2 \dif v. 
\end{align*}
In what follows, we omit the integration element $\dif v$. 

Similarly to the derivation of \eqref{DII}, a direct computation based on integration by parts and \eqref{hpsi} shows that for any $g\in C_c^\infty(\R^d)$,  
\begin{align*}
\left\langle\frac{\delta\JJ}{\delta f},\,g\right\rangle
&=\int_{\R^d} g\,\mm'\;\!|\nabla\Pe|^2 + 2\;\!\mm\nabla\Pe\cdot\nabla(g\;\!\Pe')\\
&=\int_{\R^d} g\,\mm'\;\!|\nabla\Pe|^2 - 2g\,\Delta\Pe -2g\,\Pe'\,\nabla\mm\cdot\nabla\Pe\\
&=-\int_{\R^d} g\,\mm'\;\!|\nabla\Pe|^2 + 2g\,\Delta\Pe -2g\,\mm'\;\!v\cdot\nabla\Pe. 
\end{align*}
We thus obtain the functional derivative of $\JJ$, 
\begin{align}\label{DDf}
\frac{\delta\JJ}{\delta f}=-\mm'\;\!|\nabla\Pe|^2 - 2\;\!\Delta\Pe +2\;\!\mm'\;\!v\cdot\nabla\Pe.
\end{align}

\subsection{Derivative of the Fermi-Dirac Fisher information}
We proceed to compute the evolution of $\JJ$ along solutions to \eqref{FDFP}. 
\begin{lemma}\label{lemma-FP}
For any solution $f:\R_+\times\R^d\to(0,\varepsilon^{-1})$ of \eqref{FDFP}, we have 
\begin{align}\label{FDD}
\begin{aligned}
\frac{\dif}{\dif t}\,\JJ(f)
= - \int_{\R^d} &\, 2\;\!\mm|D^2\Pe|_{\rm HS}^2 + 2\;\!\mm\mm'\;\!|\nabla\Pe|^2 -\mm''\;\!|\nabla f|^2|\nabla\Pe|^2 \\
& + 2\;\!\mm''\;\!\mm(v\cdot\nabla\Pe) \nabla f\cdot\nabla\Pe - \mm''\;\!\mm(v\cdot\nabla f) |\nabla\Pe|^2.
\end{aligned}
\end{align}
Alternatively,  this can be reformulated as
\begin{align}\label{FDDD}
\begin{aligned}
\frac{1}{2}\frac{\dif}{\dif t}\,\JJ(f)
= - \int_{\R^d} \mm|D^2\Pe|_{\rm HS}^2 + \mm(\mm'-\varepsilon\;\!\nabla f\cdot\nabla\Pe)\;\!|\nabla\Pe|^2 +2\varepsilon\;\!(\nabla f\cdot\nabla\Pe)^2 . 
\end{aligned}
\end{align}
\end{lemma}

\begin{proof}
In view of \eqref{DDf}, we have
\begin{align}\label{ddf}
\frac{\dif}{\dif t}\,\JJ(f)= \int_{\R^d} \frac{\delta\JJ}{\delta f}\, \partial_t f
=-\int_{\R^d} \left( \mm'\;\!|\nabla\Pe|^2  + 2\,\Delta\Pe -2\;\!\mm'\,v\cdot\nabla\Pe \right) \mathscr{L}_\varepsilon f. 
\end{align}
By definition of $\mathscr{L}_\varepsilon$ and integration by parts, we have  
\begin{align*}
\int_{\R^d}  \mm'\;\! |\nabla\Pe|^2\,\mathscr{L}_\varepsilon f
= \int_{\R^d} \mm'\;\!|\nabla\pe|^2\Delta f - \mm\;\! v\cdot\nabla\big(\mm'\;\!|\nabla\Pe|^2\big). 
\end{align*}
Applying \eqref{d-heat1} with $M=\R^d$ to the first term and expanding the second term on the right-hand side yields
\begin{align*}
\int_{\R^d}  \mm'\;\! |\nabla\Pe|^2\,\mathscr{L}_\varepsilon f
=\int_{\R^d}&\,\mm\Delta|\nabla\Pe|^2 - \mm''\;\!|\nabla f|^2|\nabla\Pe|^2 \\
&-2\;\!\mm\mm'\;\!v\cdot(D^2\Pe\nabla\Pe) - \mm\mm''\;\!(v\cdot\nabla f) |\nabla\Pe|^2 . 
\end{align*}
By the second equality of \eqref{LFDE} and integration by parts, we obtain 
\begin{align*}
\int_{\R^d} \left(\Delta\Pe -\mm'\;\!v\cdot\nabla\Pe\right)\mathscr{L}_\varepsilon f
&= \int_{\R^d}\mm\! \left[- \nabla \Delta\Pe + \nabla(\mm'\;\!v\cdot\nabla\Pe)\right]\cdot\nabla\Pe, 
\end{align*}
where the term 
\begin{align*}
\nabla(\mm'\;\!v\cdot\nabla\Pe)\cdot\nabla\Pe
= \mm'\;\!|\nabla\Pe|^2
+ \mm'\;\!v\cdot(D^2\Pe\nabla\Pe) +\mm''\;\!(v\cdot\nabla\Pe) \nabla f\cdot\nabla\Pe. 
\end{align*}
Collecting the preceding three identities together with \eqref{ddf}, we derive  
\begin{align*}
\frac{\dif}{\dif t}\,\JJ(f)
= - \int_{\R^d}&\, \mm \Delta|\nabla\Pe|^2 - 2\;\!\mm\nabla \Delta\Pe\cdot \nabla\Pe + 2\;\!\mm\mm'\;\!|\nabla\Pe|^2 - \mm''\;\!|\nabla f|^2|\nabla\Pe|^2 \\ 
& + 2\;\!\mm\mm''\;\!(v\cdot\nabla\Pe) \nabla f\cdot\nabla\Pe - \mm\mm''\;\!(v\cdot\nabla f) |\nabla\Pe|^2. 
\end{align*}
The desired result \eqref{FDD} then follows from Bochner's formula. 

The alternative expression \eqref{FDDD} is obtained by substituting into \eqref{FDD} the following elementary identities derived from the definitions, 
\begin{align*}
&\mm\;\!v\cdot\nabla\Pe = (\mm\nabla\Pe-\nabla f)\cdot\nabla\Pe
=\mm|\nabla\Pe|^2-\nabla f\cdot\nabla\Pe,\\
&\mm\;\!v\cdot\nabla f = (\mm\nabla\Pe-\nabla f)\cdot\nabla f
=\mm\nabla f\cdot\nabla\Pe-|\nabla f|^2. 
\end{align*}
Taking these identities and $\mm''=-2\varepsilon$ into account, we obtain \eqref{FDDD} as claimed. 
\end{proof}

\subsection{Profiles with increasing $\JJ$}
Based on the dissipation identity \eqref{FDDD}, we construct the Fermi-Dirac-type distribution for which $\JJ$ increases instantaneously at a given time under the evolution \eqref{FDFP}, analogously to part~\ref{model1} of Proposition~\ref{prop-model}.

\begin{proof}[Proof of part~\ref{FD-c} of Theorem~\ref{thm-FP}]
We consider the following initial distributions, with $\alpha>0$ and $u\in\R^d$ to be determined, 
\begin{align*}
f(v)&=\frac{1}{\varepsilon+e^{\;\!\alpha\;\!|v-u|^2/2}} ,\\
\mm(v)&=f(1-\varepsilon f) =\frac{e^{\;\!\alpha\;\!|v-u|^2/2}}{(\varepsilon+e^{\;\!\alpha\;\!|v-u|^2/2})^2},\\
\Pe(v)&=\log\frac{f}{1-\varepsilon f}+\frac{|v|^2}{2}=-\frac{\alpha\;\!|v-u|^2}{2}+\frac{|v|^2}{2}.
\end{align*}
By setting $w:=v-u$, we have
\begin{align*}
\nabla f &=-\alpha\;\!w\;\!\mm,\\
\nabla\Pe&=(1-\alpha)\;\!w+u,\\
D^2\Pe&=(1-\alpha)\;\!I_d. 
\end{align*}
By taking into account the cancellation of odd terms with respect to the variable $w$ in the integrand, a direct computation shows that each term in the integral on the right-hand side of \eqref{FDDD} from Lemma~\ref{lemma-FP} can be recast as 
\begin{align*}
\int_{\R^d} \mm|D^2\Pe|_{\rm HS}^2 &= d\,(1-\alpha)^2 \int_{\R^d} \mm , \\
\int_{\R^d} \mm\mm'\;\!|\nabla\Pe|^2 &= (1-\alpha)^2 \int_{\R^d} |w|^2\;\!\mm\mm' + |u|^2 \int_{\R^d} \mm\mm' ,\\
\int_{\R^d} \mm\nabla f\cdot\nabla\Pe|\nabla\Pe|^2 &= -\alpha\;\!(1-\alpha)^3 \int_{\R^d} |w|^4\;\!\mm^2 - \frac{d+2}{d}\,\alpha\;\!(1-\alpha)\,|u|^2 \int_{\R^d} |w|^2\;\!\mm^2,\\
\int_{\R^d} 2\;\!(\nabla f\cdot\nabla\Pe)^2 &= 2\;\!\alpha^2 (1-\alpha)^2 \int_{\R^d} |w|^4\;\!\mm^2 + \frac{2}{d}\,\alpha^2 |u|^2 \int_{\R^d} |w|^2\;\!\mm^2.
\end{align*}
Combining these identities with \eqref{FDDD} from Lemma~\ref{lemma-FP}, we obtain 
\begin{align}\label{JJDD}
\frac{1}{2}\frac{\dif}{\dif t}\bigg|_{t=0}\JJ(f)
= -(1-\alpha)^2\;\!\mathcal{D}_0 - |u|^2\;\!\mathcal{D}_1, 
\end{align}
where we collected 
\begin{align*}
\mathcal{D}_0&:= d \int_{\R^d} \mm + \int_{\R^d} |w|^2\;\!\mm\mm' +\varepsilon\;\!\alpha\;\!(1+\alpha) \int_{\R^d} |w|^4\;\!\mm^2 , \\
\mathcal{D}_1&:= \int_{\R^d} \mm\mm'  + \varepsilon\;\!\alpha\left(1-\alpha+\frac{2}{d}\right) \int_{\R^d} |w|^2\;\!\mm^2. 
\end{align*}
In particular, $\mathcal{D}_0$ and $\mathcal{D}_1$ are independent of $u$. Despite the availability of an explicit computation, it suffices to focus on the sign of $\mathcal{D}_1$. By a change of variables $z=\sqrt{\alpha}\;\!w\in\R^d$, we rewrite  
\begin{align*}
\mathcal{D}_1= \alpha^{-\frac{d}{2}} \int_{\R^d} \frac{e^{|z|^2/2}-\varepsilon }{(\varepsilon+e^{|z|^2/2})^3}\,e^{|z|^2/2}\dif z + \varepsilon\;\!\alpha^{-\frac{d}{2}}\bigg(1-\alpha+\frac{2}{d}\bigg) \int_{\R^d} \frac{|z|^2\;\!e^{|z|^2}}{(\varepsilon+e^{|z|^2/2})^4}\dif z. 
\end{align*}
It is straightforward to determine its asymptotic behaviour for large $\alpha$. We know that there are some constants $C_0,C_1>0$ depending only on $d$ and $\varepsilon$ such that
\begin{align*}
\mathcal{D}_1= C_0\;\!\alpha^{-\frac{d}{2}} -C_1\;\!\alpha^{1-\frac{d}{2}}. 
\end{align*}
This means that for any $\varepsilon>0$, we can choose $\alpha>0$ sufficiently large so that $\mathcal{D}_1<0$. Next, by taking $|u|$ large enough, we ensure that
\begin{align*}
-(1-\alpha)^2\;\!\mathcal{D}_0 -|u|^2\;\!\mathcal{D}_1>0. 
\end{align*}
We then conclude from \eqref{JJDD} that $\JJ(f)$ increases at $t=0$. 
\end{proof}

\subsection{Estimates for monotonicity of $\JJ$}
In order to derive the monotonicity of $\JJ$, we have to control the right-hand side of \eqref{FDD} in Lemma~\ref{lemma-FP} through appropriate estimates. 
\begin{lemma}\label{lemma-est}
For any solution $f:\R_+\times\R^d\to(0,\varepsilon^{-1})$ to \eqref{FDFP} satisfying 
\begin{align*}
F(t,v):=1 -2\varepsilon f -\frac{9\varepsilon}{4}|v|^2\;\!\mm\ge0, 
\end{align*}
then we have
\begin{align*}
\frac{1}{2}\frac{\dif}{\dif t}\,\JJ(f) \le - \int_{\R^d}F\;\!\mm|\nabla\Pe|^2 
\le -\left[\inf\nolimits_{\,\R^d}F(t,\cdot)\right]\JJ(f). 
\end{align*}
\end{lemma}

\begin{proof}
By the Cauchy-Schwarz inequality, we have 
\begin{align*}
2\left|\mm\left(v\cdot\nabla\Pe\right) \nabla f\cdot\nabla\Pe \right| + \left| \mm\left(v\cdot\nabla f\right)\right| |\nabla\Pe|^2
\le  \frac{9}{4}\,\mm^2|v|^2|\nabla\Pe|^2 + |\nabla f|^2 |\nabla\Pe|^2. 
\end{align*}
Combining this with \eqref{FDD} of Lemma~\ref{lemma-FP} implies that 
\begin{align*}
\frac{\dif}{\dif t}\,\JJ(f)
\le - \int_{\R^d}\left(2\;\!\mm' +\frac{9}{4}|v|^2\;\!\mm\mm'' \right) \mm|\nabla\Pe|^2. 
\end{align*}
Given that $\mm'=1-2\varepsilon f$ and $\mm''=-2\varepsilon$, the claim is established. 
\end{proof}

We are now in a position to complete the proof of Theorem~\ref{thm-FP}. It is not our intention to carry out a precise computation of the explicit bounds between the parameters $\varepsilon$ and $\beta$. In particular, the factor $1-\frac{4\varepsilon}{\beta}$ appearing in the statement of Theorem~\ref{thm-FP} is not optimal. 

\begin{proof}[Proof of part~\ref{FD-d} of Theorem~\ref{thm-FP}]
Since the initial data satisfies $0\le f(0,\cdot)\le\mbe$ for the Fermi-Dirac equilibrium $\mbe$, the comparison principle (see for instance \cite[Lemma~2.7]{CLR}) ensures this bound is preserved for the solution $f=f(t,\cdot)$ of \eqref{FDFP} for all $t\in\R_+$. It follows that, whenever $2\varepsilon\mbe\le1$, a condition that holds in particular if $2\varepsilon\le\beta$, we have 
\begin{align*}
1-2\varepsilon f-\frac{9\;\!\varepsilon}{4}|v|^2f\;\!(1-\varepsilon f)
\ge 1-2\varepsilon\mbe-\frac{9\;\!\varepsilon}{4}|v|^2\mbe(1-\varepsilon\mbe). 
\end{align*}
We simply observe from the definition \eqref{mbe} of $\mbe$ that 
\begin{align*}
2\mbe+\frac{9}{4}|v|^2\mbe(1-\varepsilon\mbe)
&\le\frac{2}{\beta}+\frac{9}{4\beta}\max\nolimits_{\;\!\R^d}|v|^2e^{-|v|^2/2}\\
&=\frac{2}{\beta}+\frac{9}{2e\beta} <\frac{4}{\beta}. 
\end{align*}
Therefore, 
\begin{align*}
\inf\nolimits_{\,\R^d} \left(1 -2\varepsilon f -\frac{9\varepsilon}{4}|v|^2\;\!\mm\right)\ge1-\frac{4\varepsilon}{\beta}, 
\end{align*}
In the light of Lemma~\ref{lemma-est}, whenever $1-\frac{4\varepsilon}{\beta}\ge0$, we derive the monotonicity of $\JJ(f)$. 
\end{proof}

\section{Linear-type Landau-Fermi-Dirac equation}\label{sec-LFP}
We examine in this section the estimates quantifying the dissipation of the Fisher information functional $\JJ$ for a linear-type Landau-Fermi-Dirac equation, with the main result stated below in Theorem~\ref{prop-landau}. The study parallels the work \cite{Villani-landau-fisher} on the (non-quantum) Landau equation with Maxwell molecules. We also note the broader relevance of the functional $\JJ$, as a weighted version has already played a role in deriving entropy dissipation estimates for the Landau-Fermi-Dirac equation in \cite{ABDL}. 

The investigation of the corresponding problem for the full Landau-Fermi-Dirac equation, which serves as the quantum analogue of the classical Landau equation addressed in \cite{Guillen-Silvestre}, will be the subject of a forthcoming study. 

The homogeneous Fermi-Dirac-Landau equation reads
\begin{align*}
\partial_t f(t,v)=Q_\varepsilon(f,f)(t,v) {\quad\rm for\quad} (t,v) \in\R_+\times\R^d. 
\end{align*}
For the case of Maxwell molecules, the collision operator $Q$, depending on the quantum parameter $\varepsilon>0$, is defined as 
\begin{align*}
Q_\varepsilon(g,h):=\nabla\cdot\int_{\R^d} A(v-w)\left[g(w)(1-\varepsilon g(w))\nabla h(v)-h(v)(1-\varepsilon h(v))\nabla g(w)\right] \dif w, 
\end{align*}
where the $d\times d$ matrix, proportional to the projection onto $(\R z)^\perp$, is given by 
\begin{align*}
A(z):=|z|^2-z\otimes z.
\end{align*}
Although the operator $Q_\varepsilon(\mbe,f)$ is not linear in $f$, it arises from a ``linearization" around the Fermi-Dirac equilibrium $\mbe$. Of interest is the resulting linear-type equation, 
\begin{align}\label{LLFD}
\partial_t f=Q_\varepsilon(\mbe,f).
\end{align}
Written out more explicitly, it is  
\begin{align*}
\partial_t f=\nabla\cdot\int_{\R^d} A(v-w)\left[\mbew(w)\nabla f(v)-\mm(v)\nabla\mbe(w)\right] \dif w, 
\end{align*}
where we use the shorthand  
\begin{align*}
\mbew:=\mbe\left(1-\varepsilon\mbe\right)=\frac{\beta\;\!e^{\;\!|v|^2/2}}{(\varepsilon+\beta\;\!e^{\;\!|v|^2/2})^2}. 
\end{align*}
In particular, the operator $Q_\varepsilon(\mbe,f)$ can be expressed in a diffusion-drift form, 
\begin{align}\label{QQ}
\begin{aligned}
Q_\varepsilon(\mbe,f)=A*\mbew:D^2 f &+ (\nabla A*\mbew)\cdot \nabla f \\
& - (\nabla A*\mbe)\cdot \nabla\mm-(D^2:A*\mbe)\,\mm. 
\end{aligned}
\end{align}

\subsection{Reformulation of the operator}
Let us compute the coefficients appearing in the above diffusion-drift form \eqref{QQ}, with the aim of verifying that  $Q_\varepsilon(\mbe,\cdot)$ splits into a spherical diffusion part and the Fermi-Dirac-Fokker-Planck part, both of which have been analysed in the preceding two sections. A closely related formulation for the classical Landau operator can be found in \cite{Villani-landau}. 

\begin{lemma}\label{lemma-LFP}
The coefficients of the linear-type Fermi-Dirac-Landau operator $Q_\varepsilon(\mbe,\cdot)$ in the expression \eqref{QQ} are given by
\begin{align}
&A*\mbew(v) = \nbew\;\!A(v) + \left(d-1\right)\nbe\,I_d, \label{coe1} \\
&\nabla A*\mbew =-\left(d-1\right)v\,\nbew, \label{coe2} \\
&\nabla A*\mbe  = -\left(d-1\right) v\,\nbe, \label{coe3} \\
&D^2:A*\mbe = -d\left(d-1\right)\nbe. \label{coe4}
\end{align}
The constants $\nbe$ and $\nbew$ appearing above denote the mass of $\mbe$ and $\mbew$, respectively, 
\begin{align*}
&\nbe:=\int_{\R^d}\mbe(v)\dif v \\
&\nbew:=\int_{\R^d}\mbew(v)\dif v. 
\end{align*}
Furthermore, the operator $Q_\varepsilon(\mbe,\cdot)$ can be recast in the form 
\begin{align}\label{QQQ}
Q_\varepsilon(\mbe,f) =  \nbew\;\!\sum\nolimits_{j<k} \Omega_{jk}^2\;\!f + \left(d-1\right)\nbe\,\mathscr{L}_\varepsilon f,  
\end{align}
where $\Omega_{jk}$, for $i,j=1,\ldots,d$, denote the generators of rotations as given in \eqref{rotation}, and $\mathscr{L}_\varepsilon$ refers to the Fermi-Dirac-Fokker-Planck operator defined in \eqref{LFDE}. 
\end{lemma}
 
\begin{proof}
To compute the leading-order coefficients, we expand $A(v-w)$ and observe that certain terms vanish due to oddness under integration. Specifically, we have 
\begin{align*}
A*\mbew(v)&= \int_{\R^d}\left[|v-w|^2-(v-w)\otimes (v-w)\right]\mbew (w) \dif w\\
&= \nbew\;\!A(v) +\int_{\R^d}A(w)\;\!\mbew(w) \dif w\\
&= \nbew\;\!A(v) +\frac{d-1}{d}\,I_d\int_{\R^d}|w|^2\;\!\mbew(w) \dif w.  
\end{align*}
By noticing $\nabla\mbe(v)=-v\;\!\mbew(v)$ and applying integration by parts, we have 
\begin{align*}
\frac{d-1}{d}\,I_d\int_{\R^d}|w|^2\;\!\mbew(w) \dif w
=\frac{d-1}{d}\,I_d\int_{\R^d}-w\cdot\nabla\mbe(w) \dif w
=\left(d-1\right)\nbe\,I_d. 
\end{align*}
Combining these two computations yields \eqref{coe1}. 

For the first-order coefficients, we use the identity $\nabla A(v)=(1-d)\;\!v$, which gives
\begin{align*}
\nabla A*\mbew = -(d-1) \int_{\R^d}\left(v-w\right)\mbew(w) \dif w =-\left(d-1\right)v\,\nbew, 
\end{align*}
This confirms \eqref{coe2}, and \eqref{coe3} is derived similarly.

The zeroth-order coefficient is found by using $D^2:A(v)=-d\;\!(d-1)$, which gives \eqref{coe4}. 

Substituting \eqref{coe1}, \eqref{coe2}, \eqref{coe3}, \eqref{coe4} into \eqref{QQ}, we obtain 
\begin{align*}
Q_\varepsilon(\mbe,f)= \nbew \left[A(v):D^2 f -\left(d-1\right)v\cdot\nabla f\right] + \left(d-1\right)\nbe\left(\Delta f+v\cdot\nabla\mm +d\;\!\mm\right). 
\end{align*}
On the right-hand side, the second term is proportional to the Fermi-Dirac-Fokker-Planck operator $\mathscr{L}_\varepsilon$. The first term corresponds to spherical diffusion in $\R^d$, which, when restricted to the unit sphere $\mathbb{S}^{d-1}$, is equivalent to the spherical Laplacian $\Delta_{\mathbb{S}^{d-1}}$. Indeed, in terms of the rotational vector fields $\Omega_{jk}=v_j\partial_{v_k}-v_k\partial_{v_j}$, one readily verifies that
\begin{align*}
\sum\nolimits_{j<k} \Omega_{jk}^2=A(v):D^2 -\left(d-1\right)v\cdot\nabla. 
\end{align*}
Therefore, we obtain the desired expression \eqref{QQQ}.
\end{proof}

\subsection{Derivative along the linear-type Landau-Fermi-Dirac flow}
Drawing on the results established in the previous two sections, we directly deduce the monotonicity properties of the Fermi-Dirac Fisher information along solutions of \eqref{LLFD}. 

\begin{theorem}\label{prop-landau}
Let $d\ge2$, and let $f:\R_+\times\R^d\to(0,\varepsilon^{-1})$ be a solution to \eqref{LLFD}. 
\begin{enumerate}[label=(\roman*), leftmargin=*, labelsep=0.5em]
\item\label{landau1} For any $\varepsilon>0$, there exists $\alpha\ge1$ and $u\in\R^d$ such that 
\begin{align*}
\frac{\dif}{\dif t}\bigg|_{t=0}\JJ(f)>0
{\quad\rm for\quad}
\left.f\right|_{t=0}\!(v)=\frac{1}{\varepsilon+e^{\;\!\alpha\;\!|v-u|^2/2}}. 
\end{align*}
\item\label{landau2} If $0\le\left.f\right|_{t=0}\le\mbe$ in $\R^d$ with the constants $\varepsilon,\beta$ satisfying $0\le6\varepsilon\le\beta$, then we have
\begin{align*}
\frac{1}{2}\frac{\dif}{\dif t}\,\JJ(f)
\le -\left(d-1\right)\nbe\left(1-\frac{6\varepsilon}{\beta}\right)\JJ(f). 
\end{align*}
\end{enumerate}
\end{theorem}

\begin{proof}
In the light of Lemma~\ref{lemma-LFP}, we decompose the time derivative of $\JJ$ along \eqref{LLFD} into the contributions arising from the spherical diffusion \eqref{model} and the Fermi-Dirac-Fokker-Planck dynamics \eqref{LFDE}, 
\begin{align*}
\frac{\dif}{\dif t}\,\JJ(f)
=\int_{\R^d} \frac{\delta\JJ}{\delta f}\, \partial_t f
=\nbew\int_{\R^d} \frac{\delta\JJ}{\delta f}\sum_{j<k}\Omega_{jk}^2\;\!f + \left(d-1\right)\nbe \int_{\R^d} \frac{\delta\JJ}{\delta f}\,\mathscr{L}_\varepsilon f.  
\end{align*}
Then part~\ref{landau1} of the statement follows immediately from part~\ref{model1} of Proposition~\ref{prop-model} and part~\ref{FD-c} of Theorem~\ref{thm-FP}. 

To establish part~\ref{landau2}, we invoke part~\ref{model2} of Proposition~\ref{prop-model} and Lemma~\ref{lemma-est}, from which it follows that
\begin{align}\label{ccc0}
\frac{1}{2}\frac{\dif}{\dif t}\,\JJ(f)
\le - \left(d-1\right)\nbe\,c_0\, \JJ(f), 
\end{align}
provided that
\begin{align*}
c_0:=\inf\nolimits_{\,\R^d} \left(1 -2\varepsilon f -\frac{9\varepsilon}{4}\;\!|v|^2\;\!\mm - \frac{\varepsilon\;\!\nbew}{\left(d-1\right)\nbe}\;\!|v|^4\;\!\mm\right)\ge0. 
\end{align*}
By the comparison principle, the solution of \eqref{LLFD} satisfies $0\le f(t,\cdot)\le\mbe$ for all $t\in\R_+$. Consequently, whenever $2\varepsilon\mbe\le1$, which is in particular guaranteed by $2\varepsilon\le\beta$, we have
\begin{align*}
c_0\ge 1-\frac{2\varepsilon}{\beta} - \frac{9\varepsilon}{4\beta}\max\nolimits_{\;\!\R^d}|v|^2e^{-|v|^2/2} -\frac{\varepsilon\;\!\nbew}{\left(d-1\right)\beta\;\!\nbe} \max\nolimits_{\,\R^d}|v|^4e^{-|v|^2/2}. 
\end{align*}
A direct computation of the maxima, as well as the fact that $\nbew \le \nbe$, implies
\begin{align*}
c_0\ge 1-\frac{2\varepsilon}{\beta} -\frac{9\varepsilon}{2e\beta}  -\frac{16}{e^2}\;\!\frac{\varepsilon}{\left(d-1\right)\beta}
> 1 - \frac{6\varepsilon}{\beta} . 
\end{align*}
Together with \eqref{ccc0}, this leads to the desired estimate. 
\end{proof}

\end{document}